\documentclass[12pt]{amsart}

\usepackage{amssymb,amscd}
\usepackage{amsbsy,amsmath,amsfonts,amsthm,bm}
\usepackage[english]{babel}
\usepackage{graphicx}
\usepackage[all]{xy}
\usepackage[latin1]{inputenc}

\newtheorem{theorem}{Theorem}[section]
\newtheorem{lemma}[theorem]{Lemma}
\newtheorem{corollary}[theorem]{Corollary}
\newtheorem{proposition}[theorem]{Proposition}

\newtheorem{GACconjecture}{Generalized Andrews-Curtis conjecture}
\newtheorem*{ACconjecture}{Andrews-Curtis conjecture}

\newenvironment{sproof}[1]
{\begin{proof}[#1]} {\end{proof}}

\newcommand{\Z}{\mathbb Z}

\newcommand{\Q}{\mathbb Q}

\newcommand{\rk}{\textnormal{rk}}

\newcommand{\BS}{\textnormal{BS}}
\newcommand{\M}{\textnormal{M}}
\newcommand{\rec}{\operatorname{rec}}
\newcommand{\AC}{\operatorname{AC}}

\newcommand{\GACC}[1]{\operatorname{GACC} #1}
\newcommand{\GACCw}[1]{\operatorname{GACC} #1_{n > w}}
\newcommand{\GACCone}[1]{\operatorname{GACC} #1_{n = 1}}
\newcommand{\diam}{\operatorname{diam}}

\newcommand{\br}[1]{\lbrack #1 \rbrack}
\newcommand{\Pres}[2]{\left\langle #1 \ \big\vert\  #2 \right\rangle}
\newcommand{\ogb}{\overline{\mathbf{g}}}
\newcommand{\ohb}{\overline{\mathbf{h}}}

\newcommand{\gb}{\mathbf{g}}
\newcommand{\hb}{\mathbf{h}}

\newcommand{\DG}{\lbrack G, G \rbrack}
\newcommand{\Gab}{G_{ab}}

\newcommand{\GW}{G/W(G)}
\newcommand{\apm}{a^{\pm 1}}
\newcommand{\bpm}{b^{\pm 1}}
\newcommand{\ncl}[1]{\langle #1 \rangle^G}

\newcommand{\ZC}{ \Z \lbrack C \rbrack}

\newcommand{\Cx}{\mathcal{C}}
\newcommand{\W}{\mathcal{W}}
\newcommand{\MN}{\mathcal{MN}}

\title{On Andrews-Curtis conjectures for soluble groups}
\address{EPFL ENT CBS BBP/HBP. Campus Biotech. B1 Building, Chemin des mines, 9\\Geneva 1202, Switzerland}
\email{luc.guyot@epfl.ch}
\author{Luc Guyot}
\date{\today}

\begin{document}
\maketitle
\begin{abstract}
The Andrews-Curtis conjecture claims that every normally generating $n$-tuple of a free group $F_n$ of rank $n \ge 2$ can be reduced to a basis by means of Nielsen transformations and arbitrary conjugations. 
Replacing $F_n$ by an arbitrary finitely generated group yields natural generalizations whose study may help disprove the original and unsettled conjecture.
We prove that every finitely generated soluble group satisfies the generalized Andrews-Curtis conjecture 
in the sense of Borovik, Lubotzky and Myasnikov. In contrast, we show that some soluble 
Baumslag-Solitar groups do not satisfy the generalized Andrews-Curtis conjecture in the sense of Burns and Macedo\'nska.
\end{abstract}

\section{Introduction}
The Andrews-Curtis conjecture ($\AC$ conjecture) is a long-standing open problem in combinatorial group theory with a tight link to open problems in low-dimensional topology \cite{Zee64, AC65, Wri75}. 
The conjecture asserts that any $n$-tuple whose components normally generate a free group of finite rank $n \ge 2$ can be transitioned to a basis by means of Nielsen transformations and arbitrary conjugations. This article, which is only concerned with algebraic aspects, addresses two generalizations of this conjecture to finitely generated soluble groups. Our main results are Theorem \ref{ThGACSoluble}, Theorem \ref{ThRecalcitranceSoluble} and Corollary \ref{CorBS} below. Theorem \ref{ThGACSoluble} elaborates on results of \cite{Mya84, BO12} to settle the generalized Andrews conjecture in the sense of \cite{BLM05} for soluble groups. Theorem \ref{ThRecalcitranceSoluble} and Corollary \ref{CorBS} solve \cite[Open Problem]{BO12} and show in particular that the metabelian Baumslag-Solitar group $\Pres{a, b}{aba^{-1} = b^{11}}$ does not satisfy the generalized Andrews conjecture in the sense of \cite{BM93}.

 Before we can state the $\AC$ conjecture with precise terms, we need to introduce some definitions.

Let $G$ be a finitely generated group. We denote by $\rk(G)$ the \emph{rank of $G$}, i.e., the minimal number of elements which generate $G$. 
For $n \ge \rk(G)$, we define a \emph{generating $n$-vector} of $G$ as an ordered $n$-tuple whose components generate $G$. 
An \emph{elementary Nielsen transformation} of $G^k$ for $k \ge 1$, is a transformation which 
replaces for some $i$ the component $g_i$ of $\gb = (g_1, \dots, g_k) \in G^k$ by either $g_j g_i$ for some $j \neq i$ or by $g_i^{-1}$ and leaves the other components unchanged. 
We call \emph{Nielsen transformation} the composition of any finite number of elementary Nielsen transformations.
The Nielsen transformations of $G^n$ for $n \ge \rk(G)$ clearly preserve the set of generating $n$-vectors of $G$. 
Two such vectors are said to be \emph{Nielsen equivalent} if they can be related by a Nielsen transformation.

 We say that a set $\{g_1,\dots, g_n\}$ of elements of a group $G$ \emph{normally generates $G$} if the normal subgroup $\langle g_1, \dots, g_n\rangle^G$ generated by these elements is $G$.
We denote by $w(G)$ the \emph{weight of $G$}, i.e., the minimal number of elements which normally generate $G$.
For $n \ge w(G)$,  we define a \emph{normally generating $n$-vector} of $G$ as an ordered $n$-tuple whose components normally generate $G$.

The elementary transformations of Andrews and Curtis \cite{AC66}, are the elementary Nielsen transformations supplemented by the transformations replacing a component $g_i$ by any of its conjugates $g_i^{g} \Doteq g^{-1} g_i g$ with $g \in G$, and leaving the other components unchanged. We refer to these transformations as the \emph{elementary $\AC$-transformations}. 
We call \emph{$\AC$-transformation} the composition of any finite number of elementary $\AC$-transformations
The $\AC$-transformations of $G^n$ for $n \ge w(G)$ clearly preserve the set of normally generating $n$-vectors of $G$. Two such vectors are said to be \emph{$\AC$-equivalent} if they can be related by an $\AC$-transformation.
The $\AC$ conjecture claims that 
\begin{ACconjecture}[\cite{AC65}, \cite{AC66}]
Every normally generating $n$-vector of a free group $F_n$ of rank $n \ge 2$ is $\AC$-equivalent to some basis of $F_n$.
\end{ACconjecture}
Since any two bases of $F_n$ are Nielsen equivalent \cite[Proposition 4.1]{LS77}, the conjecture can be rephrased saying that any two normally generating $n$-vectors of $F_n$ are $\AC$-equivalent.
We observe that, for every $n \ge 1$, an $\AC$-transformation of $G^n$ induces an $\AC$-transformation on $Q^n$ for any quotient $Q$ of $G$. Thus the existence of non-$\AC$-equivalent vectors in some quotient of $F_n$ would disprove the $\AC$ conjecture, provided these vectors can be lifted to $F_n$. This never holds for Abelian quotients:

\begin{proposition} \label{PropAbel}
Let $F_n$ be a free group of rank $n \ge 1$ and $\pi: F_n \twoheadrightarrow G$ an epimorphism onto an Abelian group $G$. Then the images under $\pi$ of any two normally generating $n$-vectors of $F_n$ are $\AC$-equivalent.
\end{proposition}

In a finitely generated Abelian group, $\AC$-transformations and Nielsen transformations coincide and define hence the same equivalence relation. Therefore the proposition holds if it is established for $G = \Z^n$, which is well-known \cite[Proposition 4.4]{LS77}. As a by-product, we obtain that the minimal number of $\AC$-transformations needed to turn one image vector into the other is bounded above by a constant which depends only on $n$.

Because of Proposition \ref{PropAbel}, none of the potential counterexamples to the Andrews-Curtis conjecture (see, e.g., \cite{HR03}) can be verified using Abelian quotients, even though finitely generated Abelian groups do have in general Nielsen non-equivalent generating vectors (consider $\Z/5\Z$ and see Theorem \ref{ThNielsenAbel} below for a complete treatment).
Borovik,  Lubotzky and Myasnikov \cite{BLM05} formulated a generalization of the $\AC$ conjecture which takes into account \emph{false positives occuring in abelianization}, i.e., pairs of normally generating vectors whose images under abelianization are not $\AC$-equivalent.
Let $\DG$ be the commutator subgroup of $G$ and let $\pi_{ab}:G \twoheadrightarrow \Gab = G/\DG$ be the abelianization homomorphism.

\begin{GACconjecture}[$\GACC{1}$] \label{GAC1}
A finitely generated group $G$ is said to satisfy the generalized Andrews-Curtis conjecture in the sense of \cite{BLM05} if the following holds 
for every $n \ge \max(w(G), 2)$: any two normally generating $n$-vectors of $G$ are $\AC$-equivalent 
if and only if their images under abelianization are $\AC$-equivalent.
\end{GACconjecture}
The question as to whether there is a finitely generated group which does not satisfy $\GACC{1}$ remains open.
The class of groups satisfying $\GACC{1}$ comprises finite groups \cite[Theorem 1.1]{BLM05}, finitely generated groups whose maximal subgroups are normal \cite[Theorem 1.5]{Myr16} and finitely generated groups of the form $S_1 \times \cdots \times S_k$ where each factor is a non-Abelian simple group \cite[Inspection of Theorem 2.1's proof]{BLM05}.
It also comprises finitely generated free soluble groups \cite[Corollary 1]{Mya84}, and more generally, finitely generated soluble groups $G$ whose abelianization homomorphism is coessential, i.e., such that every generating $n$-vector of $\Gab$ lifts to an $n$-generating vector of $G$ for every $n \ge \rk(G)$ \cite[Theorem 3.1]{BO12}. Any of the two latter results easily implies

\begin{proposition} \label{PropSoluble}
Let $F_n$ be a free group of rank $n \ge 1$ and $\pi: F_n \twoheadrightarrow G$ an epimorphism onto a soluble group $G$. 
Then the images under $\pi$ of any two normally generating $n$-vectors of $F_n$ are $\AC$-equivalent.
\end{proposition}

We shall establish with Corollary \ref{CorBS} below that there exist finitely generated soluble groups with non-coessential abelianization homomorphisms.
Therefore the aforementioned results do not fully settle $\GACC{1}$ in the variety of soluble groups. Elaborating on the techniques of \cite[Theorem 3.1]{BO12}, we are able to settle this conjecture for all soluble groups:

\begin{theorem} \label{ThGACSoluble}
Every finitely generated soluble group $G$ satisfies $\GACC{1}$:
for every $n \ge \max(w(G), 2)$, any two normally generating $n$-vectors of $G$ are $\AC$-equivalent if and only if their images under abelianization are $\AC$-equivalent.
\end{theorem}

If we restrict to the case $n > w(G)$, then the conclusion of Theorem \ref{ThGACSoluble} is actually valid for a class of groups which is larger than the union of all classes mentioned earlier. The definition of this class involves the $N$-Frattini subgroup $W(G)$ of a group $G$. The group $W(G)$ is the intersection of all maximal normal subgroups of $G$ if such exist, the group $G$ otherwise. 
This is obviously a characteristic subgroup of $G$ and it is straightforward to check that $W(G/W(G)) = 1$.
In addition, the group $W(G)$ coincides with the set of  elements which can be omitted from every normal generating subset of $G$, that is, if $\langle X, g\rangle^G = G$ for some subset $X$ of $G$ and an element $g  \in W(G)$, then $\langle X\rangle^G = G$ (see \cite[Satz 3.1]{Bae64} or \cite[Lemma 4.1.(1)]{BLM05}). As a result, we have $w(G) = w(G/W(G))$ whenever $G \neq W(G)$.
If every simple quotient of $G$ has a maximal subgroup, then $W(G)$ contains the Frattini subgroup $\Phi(G)$ of $G$, that is the intersection of all maximal subgroups of $G$ if such exist, the group $G$ otherwise \cite[Lemma 4]{DK60}. 
Thus the inclusion $\Phi(G) \subset W(G)$ holds in particular if $G$ is finitely generated or soluble.
But if $G$ is a simple group without maximal subgroups, e.g., one of those constructed in \cite{She80} or \cite[Theorem 35.3]{Ols91}, then $\Phi(G) = G$ whereas $W(G) = 1$, so that $\Phi(G) \not\subset W(G)$.
The quotient $G/W(G)$ is, by construction, a subdirect product of an unrestricted product of simple groups.
Since simple soluble groups are Abelian, the following is immediate:
\begin{lemma} \label{LemW} \cite[Folgerung 2.10]{Bae64}.
If $G$ is a finitely generated soluble group then $G/W(G)$ is Abelian.
\end{lemma}

We say that $G$ satisfies $\GACCw{1}$ if for every $n > w(G)$, any two normally generating $n$-vectors of $G$ are $\AC$-equivalent.
As finitely generated Abelian groups satisfy $\GACCw{1}$ by Theorem \ref{ThNielsenAbel}, the case $n > w(G)$ of Theorem \ref{ThGACSoluble} 
is a direct consequence of the following more general result
\begin{proposition} \label{PropW}
Let $G$ be a finitely generated group such that $G/W(G)$ satisfies $\GACCw{1}$, then so does $G$.
\end{proposition}
We have already observed, thanks to Theorem \ref{ThNielsenAbel},  that $\GACC{1}$ implies 
$\GACCw{1}$, but the question as to whether the converse holds is open.
Proposition \ref{PropW} naturally leads to the following definition. Let $\W$ be the class of the groups $G$ such that $G/W(G)$ is finite, or Abelian, or of the form $S_1 \times \cdots \times S_k$ where each factor is non-Abelian and simple. Then we have the following

\begin{corollary} \label{CorTransitive}
Let $G$ be a finitely generated group in $\W$.
Then $G$ satisfies $\GACCw{1}$.
\end{corollary}

 By construction, the class $\W$ contains all the groups which were shown to satisfy $\GACC{1}$. 
Let $\MN$ be the class of groups whose maximal subgroups are normal. 
The subclass of $\W$ which consists of the finitely generated groups $G$ such that $G/W(G)$ is Abelian naturally generalizes the class of finitely generated groups lying in $\MN$. 
Indeed $\MN$ coincides with the class of groups $G$ such that $G/\Phi(G)$ is Abelian \cite[Theorem A]{Myr15}.
Furthermore, the class $\W$ contains finitely generated group which don't sit in $\MN$ since a finitely generated soluble group in $\MN$ must be nilpotent. The class $\W$  also contains those Sunic groups of intermediate growth that don't belong to $\MN$ \cite{FG16}.

In the formulation of $\GACC{1}$, the case $n = w(G) = 1$ is purposely ignored. 
Indeed, the original $\AC$ conjecture is only concerned with $n \ge 2$ since $w(F_n) = n$ and it is trivial to address the case of $F_1$. Moreover, no non-trivial perfect group can satisfy the specialization to $n = 1$ of $\GACC{1}$ (see, e.g., \cite[Lemma 4.2]{WW78}) whereas all finite groups satisfy it when $n \ge 2$.  Still, understanding $\AC$-equivalence in the class of finitely generated weight one groups, which encompasses the long-studied class of $n$-knot groups, is a problem on its own (see, e.g., \cite{Plo83, SWW10}).
We call $g \in G$ a \emph{weight element of $G$} if $g$ normally generates $G$, i.e., $G = \langle g \rangle^G$.
We say that $G$ satisfies $\GACCone{1}$ if $w(G) = 1$ and if two weight elements of $G$ are $\AC$-equivalent whenever their images under abelianization are $\AC$-equivalent. Note that $\GACCone{1}$ is, by definition, independent of $\GACC{1}$.
In the class of finitely generated soluble groups, the following results can be established:
\begin{proposition} \label{PropNormalGenerator}
The following hold:
\begin{itemize}
\item[$(i)$]
 A finitely generated soluble group is normally generated by a single element if and only if its abelianization is cyclic. 
\item[$(ii)$]
A finitely generated metabelian group with cyclic abelianization satisfies $\GACCone{1}$.
\item[$(iii)$] There exists a three-generated polycyclic group of solubility length $3$ with infinitely many non-$\AC$-equivalent weight elements.
\item[$(iv)$]
The wreath product of finitely many finite cyclic groups with pairwise coprime orders satisfies $\GACCone{1}$.
\end{itemize}
\end{proposition}
The assertion $(i)$ of Proposition \ref{PropNormalGenerator} is a straightforward corollary of Lemma \ref{LemW}, while assertions $(ii)$ and $(iii)$ are respectively \cite[Theorem 14.1.(4)]{Hil02} and \cite[Theorem 9.1]{Hil11}. Our contribution is assertion $(iv)$ which shows that $\GACCone{1}$ can be satisfied by groups with arbitrary solubility lengths.

Here are three examples of non-cyclic metabelian groups which satisfy $\GACCone{1}$:
\begin{itemize}
\item the dihedral group of order $2p$ with $p$ an odd integer;
\item the group of the invertible affine transformations $x \mapsto ax + b$ of a finite field;
\item  the metabelian Baumslag-Solitar group
$\BS(1, 2) = \langle a, b \,\vert\, aba^{-1} = b^2 \rangle$.
\end{itemize}

Because of statements like Proposition \ref{PropSoluble} and the analogous \cite[Corollary 1.3]{BLM05} for finite groups, it is sometimes asserted that a computation in a group which satisfies $\GACC{1}$ cannot lead to a counterexample to the classical $\AC$ conjecture. This is only true if one disregards strategies based on the complexity of the reduction. Indeed, if reducing the image of a potential counterexample in a group $G$ to the image of a fixed basis requires arbitrarily many $\AC$-moves when $G$ varies is some suitable class, then the $\AC$ conjecture is disproved. 
Complexity measures tightly related to the minimal number of required $\AC$-moves were studied in \cite{Bri15, Lis15}.  
A coarser complexity measure, termed \emph{recalcitrance}, was studied in \cite{BO12, BHKM99, BM93} with a view to disprove the classical $\AC$ conjecture using reductions in free soluble groups with arbitrary soluble lengths. 

We shall define the recalcitrance of a finitely generated group in the sense of \cite{BHKM99} and present then a new generalized Andrews-Curtis conjecture based on this definition.
In \cite[Proposition 1]{BM93} it was shown that the result of any finite sequence of successive $\AC$-transformations can be achieved by means of a sequence of transformations of the form
$$
(s_1, \dots, s_n) \rightarrow (s_1, \dots, s_iw, \dots, s_n), \,w \in \ncl{s_1, \dots, s_{i - 1}, s_{i + 1}, \dots, s_n}
$$
perhaps followed by a permutation of the $n$-vector, and conversely. Such a transformation is called an $\M$-transformation ($\M$ for ``modulo'' since $s_i$ is replaced by any element congruent to it modulo the other $s_j$). We take the $\M$-transformations as our ``elementary moves" and define the \emph{recalcitrance of a normally generating $n$-vector} of a group $G$ of rank $n$ to be the least number of $\M$-transformations needed to transition from that $n$-vector into a generating $n$-vector of $G$. If such a transition is impossible, the recalcitrance is infinite. 
The \emph{recalcitrance $\rec(G)$ of the group $G$} is defined as the supremum of the recalcitrances of all of its normally generating $n$-vectors. 
The classical $\AC$ conjecture can now be rephrased in asserting that every finitely generated free group has finite recalcitrance. This leads to
\begin{GACconjecture}[$\GACC{2}$]
A finitely generated group of rank $n$ is said to satisfy the generalized Andrews-Curtis conjecture in the sense of \cite{BM93} if every normally generating $n$-tuple of $G$ has finite recalcitrance.
\end{GACconjecture}

A group $G$ of rank $n$ has recalcitrance zero if every normally generating $n$-vector of $G$ is a generating $n$-vector. 
Finitely generated groups with recalcitrance zero 
have been characterized in the following ways
\begin{theorem}\cite[Theorem 1]{BHKM99}, \cite[Theorem A]{Myr15}.
Let $G$ be a finitely generated group.
Then the following are equivalent:
\begin{itemize}
\item[$(i)$] $\rec(G) = 0$.
\item[$(ii)$] The Frattini subgroup of $G$ contains the derived subgroup of $G$.
\item[$(iii)$] Every maximal subgroup of $G$ is normal.
\end{itemize}
\end{theorem}
In particular, finitely generated nilpotent groups have recalcitrance zero while finitely generated soluble and linear groups with zero recalcitrance must be nilpotent (see \cite{BHKM99, Myr15}). As noted in \cite{Myr15}, the first Grigorchuk group and the Gupta-Sidki $p$-groups are examples of non-linear groups with recalcitrance zero. 
The infinite dihedral group is an example of group with recalcitrance one \cite[Examples 3.1]{BHKM99} .
By \cite[Theorem 2.1]{BO12}, the recalcitrance of a finite group of rank $n$ is bounded by $2n - 1$, provided that $n$ is greater than the abelianization rank. 
By \cite[Theorem 3.1]{BO12}, the recalcitrance of a soluble group $G$ of rank $n$ is bounded by $2n - 1$ if its abelianization homomorphism is coessential. 
The authors ask in \cite[Open Problem]{BO12} whether the latter condition is necessary. We answer this question in the positive and exhibit subsequently a two-generated metabelian group with infinite recalcitrance. 

\begin{theorem}[Lemma \ref{LemNormallyCoessential} and Theorem 3.1 of \cite{BO12}] \label{ThRecalcitranceSoluble}
Let $G$ be a finitely generated soluble group of rank $n$. Then $G$ satisfies $\GACC{2}$ if and only if its abelianization homomorphism is coessential. If $G$ satisfies $\GACC{2}$ then we have, furthermore, $\rec(G) \le 2n - 1$.
\end{theorem}

Using results of \cite{Guy16b}, we are able to prove with Proposition \ref{PropCoessential} below that the abelianization homomorphism of the
Baumslag-Solitar group $\BS(1, 11) = \Pres{a, b}{aba^{-1} = b^{11}}$ is not coessential.

\begin{corollary}[Proposition \ref{PropCoessential} and Corollary \ref{CorBS1N}] \label{CorBS}
 The metabelian Baumslag-Solitar group $\BS(1, 11)$ has some normally generating pairs of infinite recalcitrance. In particular, this group does not satisfy $\GACC{2}$.
\end{corollary}

Eventually, we present a graph which encodes the complexity of the reduction of normally generating vectors via $\M$-transformations.
For $n \ge w(G)$, we denote by \emph{$\M$-graph $\M_n(G)$ of $G$}, that is the graph whose vertices are the normally generating $n$-vectors of $G$ and such that two vertices are connected by an edge if an $\M$-transformation, or the inversion of one component, turns one vertex into the other.
Note that swapping two components of an $n$-vector can be carried out using three elementary Nielsen transformations followed by the inversion of a component. Therefore, replacing inversions by swaps in the definition of $\M_n(G)$ would result in a quasi-isometric graph with the same connected components.
We denote by $\diam(\Gamma)$ the diameter, possibly infinite, of a connected graph $\Gamma$. We set 
$$d_n(G) = \sup_{\Cx} \diam(\Cx)$$ where $\Cx$ ranges over the set of connected components of $\M_n(G)$.

Thanks to \cite[Theorem 3.1]{BO12} and to a careful inspection of the proofs of Proposition \ref{PropW} and 
Theorem \ref{ThGACSoluble}, we obtain

\begin{corollary}
Let $G$ be a finitely generated group in $\W$, e.g., $G$ is a finitely generated soluble group.
Then the following hold.
\begin{enumerate}
\item If $n > w(G)$, then $\M_n(G)$ is connected and we have: 
$$ \diam(\M_n(\Gab)) \le \diam(\M_n(G)) \le \diam(\M_n(\GW)) + n + w(G).$$

\item If $n = w(G) > 1$ and $G$ is moreover soluble, then the following hold.
\begin{itemize}
\item[$(i)$] The abelianization homomorphism $G \twoheadrightarrow \Gab = G/\DG$ induces bijection between the sets of connected components of  $\M_n(G)$ and $\M_n(\Gab)$.
\item[$(ii)$] We have $d_n(\Gab) \le d_n(G) \le 2 d_n(\Gab) + 8(n - 1)$.
\item[$(iii)$] \cite[Theorem 3.1]{BO12} If the abelianization homomorphism of $G$ is coessential, then a normally generating $n$-vector is a t distance at most $2n -1$ from a generating $n$-vector.
\end{itemize}
\end{enumerate}
\end{corollary}

If $G$ is a finitely generated Abelian group, then $\M_n(G)$ is the Nielsen graph $\Gamma_n(G)$ depicted in \cite{Myr16}, also known as the \emph{extended product replacement graph $\tilde{\Gamma}_n(G)$} \cite[Section 2.2]{Pak01}. 
In this case, the connected components of $\M_n(G)$ can be described using the following 
\begin{theorem} \label{ThNielsenAbel} \cite[Theorem 1.1]{Oan11}
Let $G$ be a finitely generated Abelian group whose invariant factor decomposition is
$$
\Z_{d_1} \times \cdots  \times \Z_{d_k}
$$
with $1 \neq d_1$, $d_i$ divides $d_{i + 1}$ and where $\Z_{d_i}$ stands for $\Z/d_i\Z$ with $d_i \ge 0$
(in particular $\Z_0 = \Z$).
Then every generating $n$-vector $\gb$ with $n \ge k$ is
Nielsen equivalent to
$(\delta e_1, e_2, \dots, e_k, 0, \dots, 0)$ for some $\delta \in (\Z_{d_1})^{\times}$ 
and where $e_i \in G$ is defined by $(e_i)_i = 1 \in \Z_{d_i}$ and $(e_i)_j = 0$ for $j \neq i$. 

\begin{itemize}
\item If $n > k$, then we can take $\delta = 1$.
\item If $n = k$ then $\delta$ must be $\pm \det(\gb)$ where $\det(\gb)$ is the determinant of the matrix whose rows are the images of 
the components of $\gb$ under the natural epimorphism $G \twoheadrightarrow (\Z_{d_1})^n$. 
\end{itemize}
In particular $G$ has only one Nielsen equivalence class of generating $n$-vectors for $n > k$ while it has $\max(\varphi(d_1)/2, 1)$ Nielsen equivalence classes of generating $k$-vectors where $\varphi$ denotes the Euler totient function extended by $\varphi(0) = 2$.
\end{theorem}

Under the hypotheses of Theorem \ref{ThNielsenAbel}, the diagonal action of the automorphism group of $G$ on $G^n$ induces a transitive action on the set of connected components of $\M_n(G)$, which are therefore isometric. Moreover, the diameter $d_n(G)$ of any connected component of $\M_n(G)$ is finite if and only if $G$ is finite. The understanding of $d_n(G)$ for a finite Abelian group $G$ seems to be rather limited. As an example, $d_2(\Z/p\Z) = \diam(\M_2(\Z/p\Z))$ is of order $\log(p)$ \cite[Remark 3.3]{DG99}, a result which generalizes to all finite groups \cite[Theorem 2.2.3]{Pak01}.

The paper is organized as follows. Section \ref{GACC1} is dedicated to the proofs Theorem \ref{ThGACSoluble} and Proposition \ref{PropNormalGenerator}.
Section \ref{GACC2} is dedicated to the proofs of Theorem \ref{ThRecalcitranceSoluble} and Corollary \ref{CorBS}. 

The following notation will be used throughout the article. 
Given a group homomorphism $\pi: G \rightarrow Q$ and $\gb \in G^k$, we denote by $\pi(\gb) \in Q^k$ the $k$-vector obtained by componentwise application of $\pi$ to $\gb$.
We will denote by $\pi_W$ the natural epimorphism $G \twoheadrightarrow G/W(G)$. Recall that the following are equivalent:
\begin{itemize}
\item $\gb$ normally generates $G$,
\item $\pi_W(\gb)$ normally generates $G/W(G)$.
\end{itemize}

\section{The generalized Andrews-Curtis conjecture in the sense of \cite{BLM05}} \label{GACC1}

In this section we shall establish first Theorem \ref{ThGACSoluble}, with its assumption $n \ge \max(w(G), 2)$. At the end of this section
we shall consider the more exotic case $n = w(G) = 1$. 
We start by handling the easiest range for $n$, that is $n > w(G)$. Our results for such values of $n$, i.e.,
Proposition \ref{PropW} and Corollary \ref{CorTransitive}, apply to a class wider than the class of soluble groups, namely the class $\W$ defined in the introduction.

\begin{sproof}{Proof of Proposition \ref{PropW}}
Let $n > w(G)$ and let us fix a normally generating $n$-vector $\gb = (g_1, \dots, g_n)$ of $G$. 
Since $G/W(G)$ satisfies $\GACCw{1}$ by hypothesis, the vector $\pi_{W}(\gb)$ is $\AC$-equivalent to a normally generating $n$-vector whose index $i$ component is trivial for all $i > w(G)$.
Replacing $\gb$ by an $\AC$-equivalent vector if needed, we can hence assume that $g_i \in W(G)$ for every $i > w(G)$.

Consider now another normally generating $n$-vector $\hb = (h_1, \dots, h_n)$ of $G$. We shall show that $\hb$ is $\AC$-equivalent to $\gb$.
By hypothesis, the vectors $\pi_{W}(\gb)$ and $\pi_{W}(\hb)$ are $\AC$-equivalent. 
Replacing $\hb$ by an $\AC$-equivalent vector if needed, we can therefore assume that $h_i = g_i  w_i$ with $w_i \in W(G)$ for every $i$. 
Let $i \in \{ 1, \dots, w(G)\}$.
As the vector $(h_1, \dots, h_{w(G)})$ normally generates $G$,  we can turn the component of $\hb$ with index $w(G) + 1$ into $w_i$ by means of an $\M$-transformation. Subsequently we can turn the $i$-th component, namely $h_i$, into $g_i$, using an evident Nielsen transformation. Thus we obtain a normally generating $n$-tuple $\hb'$ which is $\AC$-equivalent to $\hb$ and such that $h_i' = g_i$ for $1 \le i \le w(G)$ and $h_i' \in W(G)$ for $w(G) \le i \le n$. Since $(g_1, \dots, g_{w(G)})$ normally generates $G$, we can replace $h_i'$ by $g_i$ for every $w(G) < i \le n$ by means of $\M$-transformations, which yields $\gb$.
\end{sproof}

\begin{sproof}{Proof of Corollary \ref{CorTransitive}}
Let $G$ be a finitely generated group in $\W$ and let $H = G/W(G)$. By Proposition \ref{PropW}, it suffices to show that the conclusion holds for $H$. If $H$ is Abelian, the result follows from Theorem \ref{ThNielsenAbel}. If $H$ is finite, it follows from \cite[Theorem 1.1]{BLM05}. If $H$ is a direct product of finitely many simple non-Abelian groups, the result follows from \cite[Proof of Theorem 2.1]{BLM05}.
\end{sproof}

Because the proof  Theorem \ref{ThGACSoluble} is essentially the same when $G$ is two-generated, we shall address this case first and indicate later on how to generalize to arbitrary finite ranks. In the two-generated case, the theorem can be rephrased as follows.
\begin{proposition} \label{PropReduction}
Let $G$ be a soluble group of rank $2$ and let $(a, b)$ be a normally generating pair of $G$.
Then $(ac, bd)$ is $\AC$-equivalent to $(a, b)$ for every $c, d \in \br{G, G}$.
\end{proposition}

Proposition \ref{PropReduction} can in turn be reduced to the following
\begin{proposition} \label{PropBToTheK}
Let $G$ be a soluble group of rank $2$. Let $(a, b)$ be a generating pair of $G$ and let $k \in \Z$ be such that 
$(a, b^k)$ normally generates $G$.
Then $(ac, b^kd)$ is $\AC$-equivalent to $(a, b^k)$ for every $c, d \in \br{G, G}$.
\end{proposition}

Let us postpone the proof of Proposition \ref{PropBToTheK} to the end of this section. Taking this proposition for granted, we can now prove 
Theorem \ref{ThGACSoluble} provided $n = 2$.

\begin{sproof}{Proof of Theorem \ref{ThGACSoluble}: the case $n =2$}
Consider two normally generating pairs $\gb$ and $\hb$. Since each $\AC$-move on $G^2$ induces a Nielsen transformation on $G_{ab}^2$, the generating pairs $\pi_{ab}(\gb)$ and $\pi_{ab}(\hb)$ are Nielsen equivalent if $\gb$ and $\hb$ are $\AC$-equivalent. 
Let us show that Proposition \ref{PropReduction} implies the converse. To this end, assume that $\pi_{ab}(\gb)$ and $\pi_{ab}(\hb)$ are Nielsen equivalent and write $\gb = (a, b)$. Applying to $\hb$ a Nielsen transformation taking $\pi_{ab}(\hb)$ to $\pi_{ab}(\gb)$, we obtain a normally generating pair $\hb'$ of the form $(ac, bd)$ with $c, d \in \br{G, G}$. By Proposition \ref{PropReduction}, the pair $\hb'$, and hence $\hb$, is $\AC$-equivalent to $\gb$.

Let us now prove that Proposition \ref{PropReduction} follows from Proposition \ref{PropBToTheK}.
To do so, consider two normally generating pairs $\hb = (x, y), \hb' = (xc, yd)$ of $G$ with $c, d \in \br{G,  G}$. By Theorem \ref{ThNielsenAbel}, there is a generating pair $\gb = (a, b)$ of $G$ and $k \in \Z$ such that $\pi_{ab}(\hb)$ is Nielsen equivalent to $\pi_{ab}(a, b^k)$. Applying a suitable Nielsen transformation to $\hb$, we obtain an $\AC$-equivalent pair of the form $(ac_1, b^kd_1)$ with $c_1, d_1 \in \br{G, G}$.
Applying the same Nielsen transformation to $\hb'$, we obtain an $\AC$-equivalent pair of the form $(ac_2, b^kd_2)$ with $c_2, d_2 \in \br{G, G}$. By Proposition \ref{PropBToTheK}, the latter two pairs are both Nielsen equivalent to $(a, b^k)$. As a result, the pairs $\hb$ and $\hb'$ are $\AC$-equivalent.
\end{sproof}

The next definitions and preliminary results will enable us to present the proof of Proposition \ref{PropBToTheK} and the general case of Theorem \ref{ThGACSoluble}.
Our proofs will follow closely the lines of \cite[Theorem 3.2]{BO12}'s proof. First we need to generalize 
the concept of weighted commutators used by the latter.
Recall that $\br{G, G}$ denotes the commutator subgroup of $G$, that is the subgroup generated by the commutators $\br{x, y} \Doteq x^{-1}y^{-1}xy$ for $x, y \in G$.
Let $a$ and $b$ be arbitrary elements of a group $G$. \emph{Commutators of weight $w$ in $a$ and $b$} 
are defined inductively as follows. For $w = 1$, these are taken to be $\apm$ and $\bpm$. 
A commutator of weight $w > 1$ in $a$ and $b$ is then defined to be an element of $G$
expressible as $\br{c_1, c_2}$ where $c_1$ and $c_2$ are respectively commutators
of weights $w_1, w_2 < w$ in $a$ and $b$, such that $w_1 +  w_2 = w$. 
We define similarly \emph{commutators of weight $w$ in conjugates of $a$ and $b$}, 
these being just the conjugates of $\apm$ and $\bpm$ if $w = 1$.  
It is easy to prove, using the standard group identities
\begin{align*}
\br{xy, z} = \br{x, z}\br{\br{x, z}, y}\br{y, z} \\
\br{x,yz} = \br{x, z} \br{x, y}\br{\br{x, y}, z}
\end{align*}

that every element of the commutator subgroup of $\langle a, b \rangle$ (resp. of the commutator subgroup of $\ncl{a, b}$) can be expressed as a product of finitely many commutators of weights $\ge 2$ in $a$ and $b$ (resp. commutators in conjugates of $a$ and $b$).
The proofs of  \cite[Lemmas 3.1 and 3.2]{BO12} can be used almost verbatim to demonstrate the following
\begin{lemma} \label{LemNormalClosure}
Let $G$ be a soluble group such that $w(G) = 2$ and let $a, b, c \in G$.
 Then the following assertions hold.
\begin{itemize}
\item[$(i)$] \cite[Lemma 3.1]{BO12}
If $c$ is expressible as a weighted commutator in conjugates of $a$ and  $b$ involving $\apm$, and $g$ is any element of $G^{(i)}$, the $i$-th term of the derived series of $G$, then the result $\hat{c}$ of replacing every occurrences of $a$ in $c$ by $g$, is also in $G^{(i)}$.
\item[$(ii)$] \cite[Lemma 3.2]{BO12}
If $G = \langle a, b \rangle$ (resp. $G = \langle a, b \rangle^G$) and if $c$ is expressible as a product $c_1 \cdots c_n$ of weighted commutators $c_i$ in $a$ and  $b$ (resp. in conjugates of $a$ and $b$), of arbitrary weights $\ge 2$,
each involving $\apm$ at least twice, then $\langle ac \rangle^G$ contains $a$. 
\end{itemize}
\end{lemma}
We are now in position to prove Proposition \ref{PropBToTheK}.

\begin{sproof}{Proof of Proposition \ref{PropBToTheK}}
Let $(a, b)$ be a generating pair of $G$  and let $k \in \Z$ such that $\ncl{a, b^k} = G$.
Let $\gb = (ac, b^k d)$ with $c, d \in \DG$. We can assume that $d$ is a product of non-trivial weighted commutators of $a$ and $b$, so that each of these commutators involves $\apm$ at least once.

We shall assume first that $c$ is expressible as a product of weighted commutators in conjugates of $a$ and  $b^k$, each involving $\apm$ at least twice. By Lemma \ref{LemNormalClosure}.2, we have $a \in \ncl{ac}$. Since $G = \langle a, b \rangle$, we deduce that $\DG \subset \ncl{a}$.  Therefore we can turn $\gb$ into $(ac, b^kc)$ by means of a single $\M$-transformation. Using a subsequent Nielsen move, we get  $(ab^{-k}, b^kc)$. As $\br{G, G} \subset \ncl{ab^{-k}}$, we obtain $(ab^{-k}, b^k)$ by applying another $\M$-transformation. Applying a last Nielsen transformation yields $(a, b^k)$, which proves the result in this case.

If not all of the commutator factors $c_j$ of $c$ involve $\apm$ at least twice, we replace every occurrence of 
$b^{\pm k }$ in all of them by $d^{\mp 1}$ and expand the result as a product 
$c_{j1}c_{j2} \cdots c_{jk}$ of commutators in $a$ and $b$ using the above group identities. 
Since $d$ is a product of commutators in $a$ and $b$ involving each at least one occurrence of $\apm$, 
this procedure must produce additional occurrences of $\apm$, that is, all commutators 
$c_{jr}$ will have at least two occurrences of $\apm$. We thus obtain, via a single $\M$-transformation, a pair $(ac', b^k d)$ where $c'$ is a product of commutators in $a$ and $b$ all of which involve $\apm$ at least twice. 
Lemma \ref{LemNormalClosure}.$1.ii$ applies and proceeding now as in the earlier case, we have that at most two more $\M$-transformations to get to $(a, b^k)$.
\end{sproof}

We will prove now Theorem \ref{ThGACSoluble} in full generality.
The definition of \emph{commutators of weight $k$ in (conjugates of) $x_1,\dots,x_n$} is a straightforward generalization of the definition in the case $n = 2$ given above. The following generalization of Lemma \ref{LemNormalClosure} is immediate.

\begin{lemma} \label{LemNormalClosureN}
Let $G$ be a soluble group such that $w(G) = n$ and let $x_1, \dots, x_n, c$ be elements of $G$.
Then the following assertions hold.
\begin{itemize}
\item[$(i)$] \cite[Lemma 3.1$'$]{BO12}
If $c$ is expressible as a weighted commutator in conjugates of $x_1, \dots, x_n$, each involving $x_1^{\pm 1}$, 
and $g$ is any element of $G^{(i)}$, the $i$-th term of the derived series of $G$, 
then the result $\hat{c}$ of replacing every occurrences of $x_1$ in $c$ by $g$, is also in $G^{(i)}$.
\item[$(ii)$] \cite[Lemma 3.2$'$]{BO12}
If $G = \langle x_1, \dots, x_n \rangle$ (resp. $G = \langle x_1, \dots, x_n \rangle^G$) and if $c$ is expressible as a product $c_1 \cdots c_n$ of commutators $c_i$ in $x_1, \dots, x_n$ (resp. in conjugates of $x_1, \dots, x_n$), of arbitrary weights $\ge 2$,
each involving $x_1^{\pm 1}$ at least twice, then $\langle x_1c \rangle^G$ contains $x_1$. 
\end{itemize}
\end{lemma}

\begin{sproof}{Proof of Theorem \ref{ThGACSoluble}: the general case}
By Corollary \ref{CorTransitive}, we can assume that $n = w(G) \ge 2$.
Let $\gb$ and $\hb$ be two normally generating $n$-vectors. Since each $\AC$-move on $G^n$ induces a Nielsen transformation on $G_{ab}^n$, the generating pairs $\pi_{ab}(\gb)$ and $\pi_{ab}(\hb)$ are Nielsen equivalent if $\gb$ and $\hb$ are $\AC$-equivalent. 

We shall prove now the converse. To this end, assume that $\pi_{ab}(\gb)$ and $\pi_{ab}(\hb)$ are Nielsen equivalent. 
Reasoning as in the case $n = 2$, we can find a generating $n$-vector $(a_1, \dots, a_{n -1}, b)$ of $G$ and $k \in \Z$ such that 
$\pi_{ab}(\gb)$ and $\pi_{ab}(\hb)$ are both Nielsen equivalent to $\pi_{ab}(a_1, \dots, a_{n -1}, b^k)$. Therefore it suffices to show that 
$(a_1c_1, \dots, a_{n -1}c_{n -1}, b^kd)$ is $\AC$-equivalent to $(a_1, \dots, a_{n -1}, b^k)$ for every $c_1, \dots, c_{n -1}, d \in \DG$.

If in the expression of $c_i$ as a product of weighted commutators in conjugates of $a_1,\dots,a_{n -1}, b^k$, some of the factors do not involve $a_i^{\pm 1}$ at least twice, then we replace the occurrences of $b^k$ and $a_j \,(j \neq i)$ appearing in those factors 
with $d^{\mp 1}$ (since $b^k \equiv d^{-1} \mod \langle b^k d\rangle^G$) and the corresponding $c_j^{\mp 1}$ (since $a_j \equiv c_j^{-1} \mod \langle a_j c_j \rangle^G$), 
and repeat this until the factors contain at least two occurrences of $a_i$, or they become trivial (as elements of $G^{(k)}$, 
where $k$ is the solubility length of $G$), which will occur if $a_i^{\pm 1}$ fails to appear, 
or appears just once, in the course of this iteration.  Thus, performing at most $n - 1$ $M$-transformations if needed, we can assume that $c_i$ is expressible as a product of commutators in $a_1, \dots, a_{n -1}, b^k$, of arbitrary weights $\ge 2$, each involving $a_i^{\pm 1}$ at least twice. Since 
$\langle a_1c_1, \dots, a_{n - 1}c_{n - 1} \rangle^G$ coincides with $\langle a_1, \dots, a_{n - 1}\rangle^G$ by Lemma \ref{LemNormalClosureN}.$ii$, it contains $\DG$. Hence a subsequent $\M$-transformation can be used to obtain the vector $(a_1c_1, \dots, a_{n -1}c_{n -1}, b^kc_1)$. By means of a single Nielsen move, we get $(a_1b^{-k}, \dots, a_{n -1}c_{n -1}, b^kc_1)$. As 
$\langle a_1b^{-k}, \dots, a_{n -1}c_{n -1} \rangle^G$ still contains $\DG$, we can iterate this procedure until we obtain a vector of the form 
$(a_1b^{-k}, \dots, a_{n -1}b^{-k}, b^k)$. Eventually, $n - 1$ additional Nielsen transformations yield $(a_1, \dots, a_{n -1}, b^k)$.
\end{sproof}

We conclude this section with the proof of Proposition \ref{PropNormalGenerator} which deals with finitely generated groups of weight one. Finitely generated groups of weight at most one have been characterized as the homomorphic images of knot groups, while groups of weight at most $k$ are the homomorphic images of  $k$-complement link groups \cite{GA75}. A knot group can have infinitely many non-$\AC$ equivalent elements, each of which normally generates the whole group \cite{Plo83, SWW10}. By contrast, Proposition \ref{PropNormalGenerator} shows that the study of $\AC$-equivalence is trivial for all weight one metabelian groups and also some groups with arbitrary solubility lengths.


\begin{sproof}{Proof of Proposition \ref{PropNormalGenerator}}

$(i)$. Let $G$ be a finitely generated soluble group. By Lemma \ref{LemW}, we have $w(G) = w(G_{ab}) = w(\GW)$, a result originally proved by Baer \cite[Folgerung 2.10 and Satz 6.4]{Bae64}. Therefore $w(G) = 1$ if and only if $\Gab$ is cyclic.

$(ii)$. Let $G$ be a finitely generated metabelian group with cyclic abelianization $C = \langle c \rangle$. Let $a$ be a lift of $c$ in $G$. 
Every weight element of $G$ which is also a lift of $c$ must be of the form $ad$ for some $d \in \DG$. Therefore, it suffices to show that every $d \in \DG$ is of the form $\br{a, e(d)}$ for some $e(d) \in \DG$.
By setting $g^{c^k} \Doteq g^{a^k}$ for $g \in \DG$ and $k \in \Z$, we define an action of $C$ on $\DG$. Extending linearly this action to $\ZC$, the integral group ring of $C$, we turn $\DG$ into a module over $\ZC$. For every $k, l \in \Z$ and every $d, e \in \DG$, the following identity holds: 
$\br{a^kd, a^l e} = e^{1 - c^k}d^{-(1 - c^l)}$. Since $1 - c$ divides both $1 - c^k$ and $1 - c^l$ in $\ZC$, we have $\br{a^kd, a^l e} \in \DG^{1 - c}$. As $G = \langle a \rangle \DG$, we deduce that $\DG = \DG^{1 - c}$, which proves the result. 

$(iii)$. See  \cite[Theorem 9.1]{Hil11}.

$(iv)$.
Let us consider finite cyclic groups $C_1, \dots, C_k$. We set $G_1 = C_1$ and $G_i = C_1 \wr \cdots \wr C_i$ 
for every $1 < i  \le k$, 
so that $G_i = G_{i - 1} \wr C_i$ if $i > 1$. It is easily checked that $(G_i)_{ab} = C_1 \times \cdots \times C_i$, and the latter group is cyclic for every $i$ since the orders $\vert C_i \vert$ of the cyclic groups $C_i$ are pairwise coprime integers. By the above assertion $(1)$, we have $w(G_i) = 1$ and every weight element of $G_i$ is a lift of a generator of $(G_i)_{ab}$.
For every $i$, we pick a generator $c_i$ of $C_i$
and define a lift $x_i$ of $(c_1, \dots, c_i) \in (G_i)_{ab}$ in $G_i$ in the following way. We set $x_1 = c_1$ and $x_i = f_i c_i$ 
where $f_i \in G_{i - 1}^{C_i}$ is defined by $f_i(1) = x_{i - 1}$ and for $i > 1$ we set $f_i(c) = 1$ for $c \neq 1$. 
We shall prove by induction on $k$ that the order of $x_k$ coincides with the order of $(G_k)_{ab}$ and that $x_k$ is self-centralizing in $G_k$, i.e., the centralizer of $x_k$ in $G_k$ is the cyclic subgroup $\langle x_k \rangle$ generated by $x_k$. If $k = 1$, this is obvious. If $k > 1$, then $g_{k -1} \Doteq x_k^{\vert C_k \vert}$ belongs to $G_{k -1}^{C_k}$ and is defined by $g_{k -1}(c) = x_{k -1}$ for every $c \in C_k$. Since the order of $x_{k -1}$ is the order $(G_{k -1})_{ab}$ by induction hypothesis, we infer that the order of $x_k$ is the order of $(G_k)_{ab}$. 
As $x_{k - 1}$ is self-centralizing in $G_{k -1}$ by induction hypothesis and $c_k$ is self-centralizing in $C_k$, we deduce from \cite[Theorem 1.(5)]{KP70} that $x_k$ is self-centralizing in $G_k$, which proves the claim.
Eventually, we shall prove that every lift of $(c_1, \dots, c_k)$ in $G_k$ is a conjugate of $x_k$, hence the result. Every such lift is of the form $x_ k d$ for some $d \in \br{G_k, G_k}$. Therefore the result is established if the conjugacy class of $x_k$ has as many elements as $\br{G_k, G_k}$, that is $\vert G_k \vert / \vert (G_k)_{ab}\vert$ elements. This is indeed verified since the centralizer of $x_k$ is $\langle x_k \rangle$ and has $\vert (G_k)_{ab}\vert$ elements.

\end{sproof}

\section{The Andrews-Curtis conjecture in the sense of \cite{BO12}} \label{GACC2}
This section is dedicated to the proofs of Theorem \ref{ThRecalcitranceSoluble} and Corollary \ref{CorBS}.
We shall use the following definition in a preliminary result.
An homomorphism $\pi: G \rightarrow Q$ is termed \emph{normally coessential} if every normally generating vector of $Q$ lifts to a normally generating vector of $G$. Theorem \ref{ThRecalcitranceSoluble} is a straightforward consequence of \cite[Theorem 3.1]{BO12} and the following
\begin{lemma} \label{LemNormallyCoessential}
Let $G$ be a finitely generated group of rank $n$ and let $\pi_{ab}: G \twoheadrightarrow G_{ab}$ be its abelianization homomorphism.
\begin{itemize}
\item[$(i)$] Assume that $G$ is soluble. Let $\gb \in G^k$ with $k \ge n$. Then $\gb$ normally generates $G$ if and only if $\pi_{ab}(\gb)$ generates $G_{ab}$. In particular $\pi_{ab}$ is normally coessential.

\item[$(ii)$] 
If every normally generating $n$-vector of $G$ has finite recalcitrance and $\pi_{ab}$ is normally coessential then $\pi_{ab}$ is coessential.
\end{itemize}
\end{lemma}

\begin{proof}
$(i)$. 
Apply Lemma \ref{LemW}.
$(ii)$. Let $\ogb$ be a generating $k$-vector of $G_{ab}$ with $k \ge n$. By hypothesis there is a lift $\gb$ of $\ogb$ in $G^k$ which normally generates $G$. 
Since $\gb$ can be transitioned to a generating $k$-vector $\hb$ of $G$ by means of finitely many $\M$-transformations, there is a Nielsen transformation which takes $\ogb$ to $\pi_{ab}(\hb)$. The inverse transformation takes $\hb$ to a generating $k$-vector of $G$ which is also a lift of $\ogb$ in $G^k$. Therefore $\pi_{ab}$ is coessential.
\end{proof}

We define now a class of metabelian groups whose abelianization homomorphisms can be studied by elementary ring-theoretic means.
Let $C = \langle a \rangle$ be a cyclic group, finite or infinite, given with a generator $a$. Let $R$ be a quotient ring of $\ZC$, the integral group ring of $C$.
We denote by $\alpha$ the image of $a$ in $R$ via the quotient map; thus $\alpha \in R^{\times}$. 
Let $G = R \rtimes_{\alpha} C$ be the semi-direct product of the additive group of $R$ with $C$, where $a$ acts on $R$ via multiplication by $\alpha$. We identify $R$ and $C$ with their natural images in $G$, i.e., we have $G = RC$ and $R \cap C = 1$. 
The commutator subgroup of $G$ is easily seen to be $(1 - \alpha )R$. 
Then we have $G_{ab} = R_C \times C$ where $R_C \Doteq R/(1- \alpha)R$ is a homomorphic image of $\Z$.
Moreover the abelianization homomorphism 
$\pi_{ab}: G \twoheadrightarrow G_{ab}$ splits as $\pi_{ab} = \pi_C \times Id_C$ where $\pi_C: R \twoheadrightarrow R_C$ is the natural map and 
$Id_C$ denotes the identity endomorphism of $C$. 

\begin{proposition} \label{PropSurjection}
Let $G = R \rtimes_{\alpha} C$ be as above. 
Then the following are equivalent:
\begin{itemize}
\item[$(i)$] The abelianization homomorphism $G \twoheadrightarrow G_{ab}$ is coessential. 
\item[$(ii)$] The natural group homomorphism $R^{\times} \rightarrow R_C^{\times}$ is surjective.
\end{itemize}
\end{proposition}

\begin{proof}
Since $G$ is two-generated, the introductory remark of \cite{Oan13} implies that $\pi_{ab}$ is coessential if and only if every generating pair of $G_{ab}$ lifts to a generating pair of $G$.
Let $\ogb$ be a generating pair of $G_{ab}$. By \cite[Corollary 1]{Guy16b}, there is a generating pair $\ohb$ which is Nielsen equivalent to $\ogb$ and of the form $(u, a^k)$ for some $u \in R_C^{\times}$ and some $k \in \Z$ such that $C = \langle a^k \rangle$. 
Clearly $\ogb$ lifts to a generating pair of $G$ if and only if $\ohb$ does.
By \cite[Lemma 7]{Guy16b}, the pair $\ohb$ lifts to a generating pair of $G$ if and only if $u$ lifts to a unit of $R$. 
\end{proof}

Thus the abelianization homomorphism of $R \rtimes_{\alpha} C$ is certainly coessential if $R_C \simeq \Z/n\Z$ with $n \in \{0, 1, 2, 3, 4, 6\}$. This also clearly holds if $1 - \alpha$ is a nilpotent element of $R$, i.e., $(1 - \alpha)^n = 0$ for some $n > 0$, which means that $G$ is nilpotent. If $G$ is nilpotent then the result is actually well-known and generalizes as follows: a
group $G$ with zero recalcitrance has a coessential abelianization homomorphism \cite[Proposition 2.2]{Myr16}.

\begin{proposition} \label{PropCoessential}
Let $R$ and $R_C$ be as above. 
Then the natural group homomorphism $R^{\times} \rightarrow R_C^{\times}$ is surjective in the following cases:
\begin{itemize}
\item[$(i)$] $R$ is finite.
\item[$(ii)$] $R = \Z_m\br{X^{\pm 1}}$, the ring of Laurent polynomials over $\Z_m \Doteq \Z/m\Z$ and $\alpha = X$.
\item[$(iii)$] $R = \Z\br{\zeta_n}$  where $\alpha = \zeta_n = e^{\frac{2i \pi}{n}}$ and $n > 0$.
\end{itemize}
\end{proposition}

\begin{proof}
$(i)$. Apply Justin Chen's lemma, see \cite[Lemma 16]{Guy16b} or \cite{mathoverflow}.

$(ii)$. Any $u \in (\Z_m)^{\times}$ is the image of the corresponding constant Laurent polynomial which is clearly invertible in $\Z_m\br{X^{\pm 1}}$.

$(iii)$. If $n$ has more than one prime divisor, then $1 - \zeta_n$ is a unit of $\Z\br{\zeta_n}$ by \cite[Proposition 7.6.2.$ii$]{Wei98} and hence $R_C$ is trivial. So we can assume that $n$ is a power of a prime. In this case $\xi_a = \frac{1 - \zeta_n^a}{1 -\zeta_n}$ is a unit for every $a \in \Z$ coprime with $n$ by \cite[Proposition 7.6.2.$ii$]{Wei98}. We have $R_C = \Z/\Phi_n(1) \Z$ where $\Phi_n$ is the $n$-th cyclotomic polynomial and we know that $\Phi_n(1)$ divides $n$. As $\xi_a \equiv a \mod (1 - \zeta_n) \Z\br{\zeta_n}$, the result follows.
\end{proof}

\begin{corollary} \label{CorBS1N}
The Baumslag-Solitar group 
$\BS(1, n) = \Pres{a,b}{aba^{-1} = b^n}$
has infinite recalcitrance if $n = p^d$ with $p$ prime, $d \ge 1$ and $n \ge 11$.
\end{corollary}

\begin{proof}
It is well-known that $\BS(1, n) \simeq R \rtimes_{\alpha} C$ with $R = \Z\br{1/n}$, $C = \Z$ and $\alpha = n$. We have $R_C = R/(n - 1)R \simeq \Z/(n - 1) \Z$. 
The unit group of $R$ is the multiplicative subgroup of $\Q^{\ast}$ generated by $-1$ and the prime divisors of $n$. If $n = p^d$ then this group maps onto a cyclic subgroup of  $R_C^{\times}$ of order at most $2d$. For $n \ge 11$, the group $R_C^{\times}$ has strictly more than $2d$ elements, so that $R^{\times} \rightarrow R_C^{\times}$ is not surjective. 
The result now follows from Theorem \ref{PropSurjection}
\end{proof}

\paragraph{\textbf{Acknowledgements}.}
The author is grateful to Tatiana Smirnova-Nagnibeda for pointing out that some Sunic groups of intermediate growth sit in the class $\W$ defined in the introduction.
The author is also thankful to Wolfgang Pitsch, Pierre de la Harpe, Yves de Cornulier, Daniel Oancea and Robert Geoffrey Burns for their helpful comments and their encouragements.

\bibliographystyle{abbrv}
\makeatletter
\renewcommand\@biblabel[1]{#1.}
\makeatother
\bibliography{Biblio}

\end{document}